\documentclass[11pt]{amsart}
\usepackage{amsmath,amssymb,amscd}
\usepackage{epsfig}
\usepackage{epsfig, graphicx,pinlabel}

\newtheorem{theorem}{Theorem}

\newtheorem{corollary}{Corollary}

\newtheorem*{thma}{Theorem}

\theoremstyle{definition}

\newtheorem{definition}{Definition}

\def\bN{\mathbb N}

\def\bQ{\mathbb Q}

\input xy
\xyoption{all}

\input epsf
\begin{document}

\title[Graphs with given weighted number of components]{Enumeration of graphs with given  weighted number of  connected components}
\date{\today}
\author{Joungmin Song}
\address{Division of Liberal Arts \& Sciences\\GIST\\ Gwangju, 500-712, Korea}

\email{songj@gist.ac.kr}
\begin{abstract} We give a generating function for the number of graphs with given numerical properties and prescribed weighted number of connected components. As an application, we give a generating function for the number of bipartite graphs of  given order, size and number of connected components.

\end{abstract}
\maketitle

\section{Introduction}
In this paper, we consider the generating function for the number of graphs with prescribed numerical properties and the {\it weighted number} of connected components: Given a weight vector $\omega = ( \omega_1, \omega_2, \dots)  \in \bQ^\infty$, the {\it $\omega$-weighted number of connected components} of a graph $G$ with connected components $G_1, \dots, G_s$ is defined to be
\[
h_0^\omega(G) := \sum_{i=1}^s \omega_{|G_i|} |G_i|.
\]
The notion of weighted number of components frequently arises in everyday life, such as group or bundle discount.  More complex and sophisticated applications may be found in network analysis.

In this paper, we give a generating function for the number of graphs of given order, size, and $h_0^\omega$. (In fact, order, size can be replaced by homogeneous properties). When $\omega$ is the uniform trivial weight $ (1, 1, 1, \dots)$, $h^\omega_0$ is simply the number of connected components. Our method is a slight modification of the exponential formula \cite{StanleyV2}. We first define an auxiliary multi-variabled exponential generating function Equation~(\ref{E:aux}) which enumerates the number of graphs with prescribed number of connected components of given order, and use a ring homomorphism $\tau_\omega$ to induce the desired generating function. Let $f_i$ be additive functions on graphs (Definition~\ref{D:additive}) and $\mathcal P$ be a collection of homogeneous properties (Definition~\ref{D:homogen}).

\begin{thma} The number of graphs with properties $\mathcal P$, given $f_i$ values and $h_0^\omega$ is generated by
\[
\tau_\omega
\left(\exp \sum_{n, k_i} g_{n,k_1,\dots,k_s} \frac1{n!} x^n \prod_{i=1}^s y_i^{k_i}z_n\right)
\]
where  $g_{n,k_1,\dots,k_s}$ is the number of  connected graphs with properties $\mathcal P$ and given numerical values of $f_i$, $i=1,\dots, s$, and $\tau_\omega$ is the ring homomorphism determined by mapping $ \prod z_i^{\alpha_i}$ to $z^{\sum \omega_i\alpha_i}$ (Definition~\ref{D:tau}).

\end{thma}
\

As an application, we shall enumerate the number $\beta_{n,k,\nu}$  of bipartite graphs of order $n$, size $k$ with $\nu$ connected components.  We specifically chose this trivial weight $\omega = (1,1,1,\dots)$ case because of the importance of bipartite graphs and because it will be used in our forthcoming paper \cite{SONG4}. Enumeration of bipartite graphs have been studied by many authors (\cite{Harary58, Harary63, Harary73, Hanlon, StanleyV2, Ueno} to name just several), but our search did not turn up a table of $\beta_{n,k,\nu}$, and we put it in the appendix.

All graphs are assumed to be labeled.

\section{Exponential formula}
In this section, we quickly review the exponential formula for generating functions \cite{StanleyV2} and give a few variant forms of it that will be useful for our purpose here and in \cite{SONG4}.

\begin{definition}\label{D:additive} We say that an integer valued function $f$ on the set of graphs is {\it additive} if $f(G + G') = f(G) + f(G')$ for any two graphs $G, G'$ with disjoint vertex sets.
\end{definition}

For example, order, size and rank (of the incidence matrix) are all additive.

\begin{definition}\label{D:homogen} We say that  a graph property $P$ is {\it homogeneous} if $G$ satisfies $P$ then all of its components satisfy $P$.
\end{definition}

Let $f_1, \dots, f_s$ be additive functions on graphs. Say one is interested in enumerating graphs with given order and given values of $f_i$, $i = 1, \dots, s$. Then the exponential formula (see for instance, \cite{StanleyV2} allows one to relate the number of graphs with the desired numerical properties and the number of connected graphs with the same properties. Let $[n] = \{1,2, \dots, n\}$.

\begin{theorem}(\cite[Theorem~1]{Read}, \cite{StanleyV2}) \label{T:Read} Let $P$ be a homogeneous graph property. Let $g_{n,k_1,\dots,k_s}$ (resp. $\bar g_{n,k_1,\dots,k_s}$) be the number of (resp. connected) graphs $G$ on $[n]$ with property $P$ and $f_i(G) = k_i$, $\forall i$. Then the generating function for $g_{n,k_1,\dots, k_s}$ is given by
\[
g(x,y_1,\dots,y_s) = \exp\left(\sum \bar g_{n,k_1,\dots,k_s} \frac1{n!} x^n y_1^{k_1}\cdots y_s^{k_s} \right)
\]
or equivalently, the generating function for $\bar g_{n,k_1,\dots,k_s}$ is given by the formal logarithm
\[
\overline g(x,y) = \log \left(\sum g_{n,k_1,\dots,k_s} \frac1{n!} x^n y_1^{k_1}\cdots y_s^{k_s} \right).
\]
\end{theorem}
By using this theorem, one can also sort out unwanted components when enumerating graphs. Suppose that the generating function $g(x,y_1,\dots,y_s)$ as in Theorem~\ref{T:Read} is known and we want to enumerate the graphs with the same numerical features but without certain components.

\begin{corollary}\label{C:remove-components} Let $\mathcal I \subset \mathbb Z^{s+1}$. The number of graphs of given order and given values of $f_i$ without a connected component of order $n^*$ and $f_i$ value of $k_i^*$ is generated by
\[
\exp\left( \log \left(g(x,y_1,\dots,y_s) - \sum_{(n^*,k_i^*) \in \mathcal I}\overline g_{n^*,k_1^*,\dots,k_s^*}\frac{1}{n^*!}x^{n^*}y_1^{k_1^*}\cdots y_s^{k_s^*} \right) \right)
\]
\end{corollary}
\begin{proof} By construction,  the $x^n\prod y_i^{k_i}$-coefficient comes from the partitions of $n$ and $k_i$ that do not involve $n^*$ and $k_i^*$. The assertion follows immediately.
\end{proof}
In particular, the number of graphs with desired numerical properties but without an isolated vertex is generated by
$\exp\left( \log \left(g(x,y_1,\dots,y_s) - x\right)\right)$.

\section{Graphs with a given number of weighted connected components}\label{S:main}
Suppose we have the generating function $g(x,y) = \sum_{n\ge 1,k_i} g_{n,k_1,\dots,k_s} \frac1{n!} x^n \prod_{i=1}^s y_i^{k_i}$ for the number of  connected graphs with certain homogeneous properties $\mathcal P$ and given numerical values of $f_i$, $i=1,\dots, s$.

To control the weighted number of connected components, we define an auxiliary generating function
\begin{equation}\label{E:aux}
\exp \, \tilde g(x,y,z) =\exp \sum_{n, k_i} g_{n,k_1,\dots,k_s} \frac1{n!} x^n \prod_{i=1}^s y_i^{k_i}z_n \in \bQ[x,y_1, \dots, y_s, z_1, z_2, z_3\dots, ]
\end{equation}
which involves infinitely many variables $z_i$, $i \in \bN$.
The upshot is that the auxiliary $z_i$'s allow one to keep track of the orders of connected components.

\begin{definition} \label{D:tau} Let $\omega  = (\omega_1, \omega_2, \omega_3, \dots) \in \bQ^\infty$ and define
$\tau_\omega: \bQ[z_1, z_2, z_3, \dots] \to \overline{\bQ(z)}$ to be  the ring homomorphism determined by
\[
z^\alpha := \prod z_i^{\alpha_i} \mapsto z^{\omega.\alpha} = z^{\sum \omega_i\alpha_i}.
\]
\end{definition}
Here, $\overline{\bQ(z)}$ denotes the algebraic closure of the function field $\bQ(z)$, employed just to make sure that $\tau_\omega$ is a well-defined ring homomorphism. In practice, computations involving the fractional powers of $z$ are done purely symbolically.

\begin{theorem} \label{T:enumeration-nu}
The number of graphs with properties $\mathcal P$, given $f_i$ values and $h_0^\omega$ is generated by
$
\tau_\omega (\exp \tilde g).
$
That is, if
\[
\tau_\omega (\exp \tilde g) = \sum  T_{n,k,\nu} \frac{1}{n!} x^n y^k z^\nu, \quad y^k = \prod_{i=1}^s y_i^{k_i}
\]
 then $T_{n,k,\nu}$ is precisely the number of graphs with properties $\mathcal P$, $f_i$-value $k_i$ and $h^0_\omega = \nu$.
\end{theorem}

\begin{proof} We shall give a proof for the case $s = 1$ (single numerical feature).  One will see immediately that the general case can be proved in the exact same fashion, but it is just more cumbersome to write.
Fix $\nu$ and choose a set of partitions
\begin{equation}\label{E:partition}
\begin{array}{c}
n = \sum_{i=1}^\ell (\sum_j m_{ij}) n_i, \quad k = \sum_{i,j} m_{ij} k_{ij}, \quad \mbox{such that $\nu = \sum_i (\omega_i \sum_j m_{ij})$},
\\
\mbox{$n_1 < n_2 < \cdots < n_\ell $ and $k_{i1} < k_{i2} < \cdots $, $\forall i$.}
\end{array}
\end{equation}
We first count the number of graphs with the properties $\mathcal P$ which has precisely $m_{ij}$  connected components of order $n_i$ and numerical value $k_{ij}$. We first choose $m_i = \sum_j m_{ij}$ many sets $V_{ijr}$ of vertices of cardinality $n_i$ out of $[n]$.
Note that we have triply indexed the vertex sets since for each fixed $i$ and $j$, we need to choose $m_{ij}$ many sets of vertices i.e. $V_{ij1}, V_{ij2}, \dots, V_{ijm_{ij}}$.
 There are $\binom n{\underbrace{n_1 \cdots n_1}_{m_1} \cdots \underbrace{n_\ell \cdots n_\ell}_{m_\ell}}$ ways to do this.
 Once we have chosen $V_{ijr}$, on each $V_{ijr}$, there are $g_{n_i,k_{ij}}$ many ways to draw the graphs (connected components) on $V_{ijr}$ with the desired properties.
If the last index of $V_{ij1}, \dots, V_{ijm_{ij}}$  mattered, then we would simply have $g_{n_i,k_{ij}}^{m_{ij}}$ many graphs on $V_{ijr}$.
It does not, so we divide out by $m_{ij}!$ to remove the repetition.
All in all, there are
\begin{equation}\label{E:number-components}
\sum \binom n{\underbrace{n_1 \cdots n_1}_{m_1} \cdots \underbrace{n_\ell\cdots n_\ell}_{m_\ell}} \frac1{\prod_{i,j} m_{ij}!} \prod g_{n_i,k_{ij}}^{m_{ij}}
\end{equation}
many graphs with $\ell$ components, where the sum runs over all partitions of $n$ and $k$ as in Equation~(\ref{E:partition})(note that $\ell$ is fixed here).

On the other hand, consider $\exp \tilde g = \prod \exp\left( g_{n,k} \frac1{n!} z_n y^kx^n\right)$.
We shall compute its $\prod (z_{n_i}y^{k_{ij}}x^{n_i})^{m_{ij}}$ coefficient where $\{n_i, k_{ij}\}$ gives a partition as in Equation~(\ref{E:partition}).
By definition of $\prod \exp\left( g_{n,k} \frac1{n!} z_n y^k x^n\right)$, the product is over the partitions Equation~(\ref{E:partition}). Consider a general term
\[
\prod_{u,v} \left(z_{n'_u} y^{k_{uv}'} x^{n'_u} \right)^{m_{uv}'}
\]
coming from a partition $n = \sum_{u,v} m'_{uv}n'_u$ and $k = \sum_{u,v} m'_{uv} k'_{uv}$.
Suppose that it is equal to $\prod (z_{n_i} y^{k_{ij}}x^{n_i})^{m_{ij}}$.
Then by considering the $z_n$ factor, we conclude that for each $u$ and $v$, $n'_u = n_i$ and $m'_{uv} \le m_{ij}$.
By symmetry,  we conclude that the two sets of partitions $(n = \sum m_{ij}n_i, k = \sum m_{ij} k_{ij})$ and $(n = \sum m'_{uv}n'_u, k = \sum m'_{uv}k'_{uv})$ are the same.
Hence the $\prod (z_{n_i}y^{k_{ij}}x^{n_i})^{m_{ij}}$ coefficient is simply the product of the coefficients of $z_{n_i}y^{k_{ij}}x^{n_i}$:
\begin{equation}\label{E:general-term}
\prod \frac1{m_{ij}!}  \left(\frac1{(n_i)!} g_{n_i,k_{ij}}\right)^{m_{ij}} = \frac1{(n_1!)^{m_1} \cdots (n_\ell!)^{m_\ell}} \prod \frac1{m_{ij}!} \prod g_{n_i,k_{ij}}^{m_{ij}}.
\end{equation}


Let $\tau_\omega$ be as in Definition~\ref{D:tau} but extended to the rings with $\bQ$ replaced by $\bQ[x,y_1,\dots,y_s]$.
We have
\[
\tau_\omega\left(\prod (z_{n_i} y^{k_{ij}} x^{n_i})^{m_{ij}}\right) = z^{\sum_i (\omega_i \sum_jm_{ij})} \prod y^{\sum m_{ij} k_{ij}} x^{\sum m_{ij}n_i}= z^{\sum_i (\omega_i \sum_j m_{ij})}y^kx^n.
\]
Hence the $x^n y^k z^\nu$-coefficient of $\tau_\omega(\exp\tilde g)$ is
\[
\sum \left(\mbox{$\prod (z_{n_i} y^{k_{ij}}x^{n_i})^{m_{ij}}$ coefficient of $\exp\tilde g$} \right)
\]
where the sum runs over all partitions as in Equation~(\ref{E:partition}).
Substituting Equation~(\ref{E:general-term}), we obtain
\[
T_{n,k,\nu} = \sum \frac{n!}{(n_1!)^{m_1} \cdots (n_\ell!)^{m_\ell}} \prod \frac1{m_{ij}!} \prod g_{n_i,k_{ij}}^{m_{ij}}
\]
which precisely equals Equation~(\ref{E:number-components}).
\end{proof}

\section{Application: bipartite graphs of given order, size and rank}

Our motivation for studying the enumerative combinatorics of bipartite graphs comes from certain hyperplane arrangements \cite{SONG1,SONG2}, but they are certainly interesting on their own \cite{Harary58, Harary63, Harary73, Hanlon, Ueno}. Following these earlier works, we shall first enumerate bi-colored graphs and from it,  we shall obtain a method for enumerating bipartite graphs of given order and size.

\begin{definition} A graph is {\it bi-colored} if its vertices are colored black and white so that no two adjacent vertices have the same color.
\end{definition}

The number of bi-colored graphs of order $n$ and size $k$ is easy to enumerate: any such graph has precisely two blocks of orders, say $i$ and $n-i$. Assume one is colored black and the other, white. There are $i(n-i)$ possible edges between the two blocks. Hence the total number of bi-colored graphs of order $n$ and size $k$ is
\[
\sum_{i = 0}^n \binom ni\binom{i(n-i)}{k}.
\]
The number of {\it connected} bi-colored graphs of given order and size is then generated by the formal logarithm
\[
\log \left(1 + \sum_{n\ge 1,k\ge 0} \left(\sum_i \binom ni\binom{i(n-i)}k \frac1{n!} x^n y^k\right)\right).
\]
Since a connected bipartite graph admits precisely two bi-colorings, the number of connected bipartite graphs of given order and size is exactly half of the  number of connected bi-colored graphs of the same order and size. That is, if we let
 $\bar b_{n,k}$ denote the number of connected bipartite graphs of order $n$ and size $k$, then we have
 \[
\mathcal B(x,y) :=  \sum_{n\ge 0, k\ge 0} \bar b_{n,k} \frac1{n!}x^ny^k = \frac12 \log \left(
 1+ \sum_{n\ge 1,k\ge 0} \left(\sum_i \binom ni\binom{i(n-i)}k \frac1{n!} x^n y^k\right)\right).
 \]
By Theorem~\ref{T:enumeration-nu}, we have
\begin{corollary}\label{C:bipartite} The number of bipartite  graphs of given order, size and number of connected components is given by
$
\tau_{\omega} (\exp \tilde B(x,y,z))
$
with $\omega = (1,1,\dots)$.
That is, if
\[
\tau_\omega (\exp \tilde{\mathcal B}) = \sum  b_{n,k,\nu} \frac{1}{n!} x^n y^k z^\nu,
\]
 then $b_{n,k,\nu}$ is precisely the number of bipartite graphs of order $n$, size $k$ with $\nu$ connected components.\end{corollary}
The generating function $\tau_\omega(\exp \tilde B)$ can be easily computed by computer algebra systems such as Mathematica. In the subsequent section, we shall list the first $100$ or so terms of the generating function (order of the graph up to 10).

\appendix
\section*{Appendices}
\addcontentsline{toc}{section}{Appendices}
\renewcommand{\thesubsection}{\Alph{subsection}}

\subsection{Numerical results}
\subsubsection{The generating function for the number of bipartite graphs of given order, size and number of connected components (Corollary~\ref{C:bipartite}) as computed by Mathematica}

\small
\[
\begin{array}{l}
\tau_\omega(\exp \tilde{\mathcal B}) =  1+ \frac{1}{2} x^2 y z + \frac{1}{2} x^3 y^2 z + x^4 \left(\frac{y^2 z^2}{8}+\frac{2 y^3 z}{3}+\frac{y^4 z}{8}\right) +x^5 \left(\frac{y^3 z^2}{4} +\frac{25 y^4
   z}{24} +\frac{y^5 z}{2}+ \frac{y^6 z}{12}\right) + \\
 x^6 \left(\frac{y^3 z^3}{48} +\frac{11 y^4 z^2}{24} +y^5
   \left(\frac{z^2}{16}+\frac{9 z}{5}\right) +\frac{19 y^6 z}{12}+ \frac{2 y^7 z}{3} +\frac{7 y^8 z}{48}+ \frac{y^9 z}{72}\right)+\\
x^7
   \left(\frac{y^4 z^3}{16} +\frac{41 y^5 z^2}{48}  +y^6 \left(\frac{5 z^2}{16}+\frac{2401
   z}{720} \right) +y^7 \left(\frac{z^2}{24}+\frac{55 z}{12}\right)  +\frac{10 y^8 z}{3}  +\frac{37 y^9 z}{24}  +\frac{37 y^{10} z}{80}\right)+ \\
 x^8 \left(\frac{z y^{16}}{1152}+\frac{11 z y^{15}}{720}+\frac{z y^{14}}{8}+\frac{91 z y^{13}}{144}+\frac{1583 z y^{12}}{720}+\frac{1327 z
   y^{11}}{240}+\left(\frac{z^2}{144}+\frac{491 z}{48}\right) y^{10} + \right. \\
 \left. \left(\frac{7 z^2}{96}+\frac{124 z}{9}\right) y^9+\left(\frac{49 z^2}{128}+\frac{611 z}{48}\right) y^8+\left(\frac{9
   z^2}{8}+\frac{2048 z}{315}\right) y^7+\left(\frac{z^3}{64}+\frac{1183 z^2}{720}\right) y^6+\frac{7 z^3 y^5}{48}+\frac{z^4 y^4}{384}\right) +\\
x^9   \left(\frac{z y^{20}}{2880}+\frac{z y^{19}}{144}+\frac{143 z
   y^{18}}{2160}+\frac{2 z y^{17}}{5}+\frac{2471 z y^{16}}{1440}+\frac{133 z y^{15}}{24}+\frac{140281 z y^{14}}{10080}+\frac{1}{9} \left(\frac{z^2}{32}+\frac{9947 z}{40}\right) y^{13}+ \right. \\
   \frac{1}{9}
   \left(\frac{3 z^2}{8}+\frac{31357 z}{80}\right) y^{12}+\frac{1}{9} \left(\frac{343 z^2}{160}+\frac{7779 z}{16}\right) y^{11}+\frac{1}{9} \left(\frac{123 z^2}{16}+\frac{14763 z}{32}\right)
   y^{10}+\frac{1}{9} \left(\frac{305 z^2}{16}+\frac{24903 z}{80}\right) y^9 +\\
\left.   \frac{\left(140 z^3+47670 z^2+177147 z\right) y^8}{13440}+\frac{\left(135 z^3+4697 z^2\right) y^7}{1440}+\frac{61 z^3
   y^6}{192}+\frac{z^4 y^5}{96}\right) + \\
x^{10}    \left(\frac{1771 z y^{20}}{720}+\frac{4129 z y^{19}}{480}+\frac{2927 z y^{18}}{120}+\frac{1}{10} \left(\frac{5
   z^2}{1152}+\frac{164105 z}{288}\right) y^{17}+\frac{1}{10} \left(\frac{11 z^2}{144}+\frac{2976223 z}{2688}\right) y^{16}+ \right.\\
   \frac{1}{10} \left(\frac{5 z^2}{8}+\frac{339671 z}{189}\right)
   y^{15} +
  \frac{1}{10} \left(\frac{115 z^2}{36}+\frac{38943 z}{16}\right) y^{14}+\frac{1}{10} \left(\frac{1097 z^2}{96}+\frac{98255 z}{36}\right) y^{13}+ \\
  \frac{1}{10} \left(\frac{52303
   z^2}{1728}+\frac{89665 z}{36}\right) y^{12}+\frac{1}{10} \left(\frac{z^3}{96}+\frac{3661 z^2}{72}+\frac{1}{8} \left(\frac{z^2}{18}+\frac{491 z}{6}\right) z+
   \frac{16070 z}{9}\right) y^{11}+ \\
   \frac{1}{10}
   \left(\frac{7 z^3}{64}+\frac{5783 z^2}{72}+\frac{1}{8} \left(\frac{7 z^2}{12}+\frac{992 z}{9}\right) z+\frac{16709 z}{18}\right) y^{10}+\frac{1}{10} \left(\frac{17 z^3}{32}+\frac{5743
   z^2}{72}+\frac{1}{8} \left(\frac{49 z^2}{16}+\frac{611 z}{6}\right) z+ \right.\\
 \left.  \frac{156250 z}{567}+\frac{5}{48} \left(z^3+50 z^2\right)+\frac{1}{16} \left(z^3+110 z^2\right)\right) y^9+\frac{\left(15330
   z^3+268559 z^2\right) y^8}{40320}+\frac{\left(15 z^4+3916 z^3\right) y^7}{5760}+\\
 \left. \frac{17 z^4 y^6}{576}+\frac{z^5 y^5}{3840}\right) + \cdots
\end{array}
\]

\small


\begin{table}[]
\centering
\caption{Numer of bipartite graphs of given (order, size, number of components)}
\label{my-label}
\begin{tabular}{|c|c|c|c|c|c|c|c|c|c|c|c|c}
\hline
2,1,1       & 3,2,1       & 4,3,1      & 4,4,1      & 4,2,2      & 5,4,1     & 5,5,1  \\   \hline
1           & 3           & 16         & 3          & 3          & 125       & 60      \\   \hline  \hline
5,6,1   & 5,3,2   & 6,5,1    & 6,6,1       & 6,7,1       & 6,8,1      & 6,9,1     \\  \hline
 10      & 30      & 1296     & 1140        & 480         & 105        & 10      \\  \hline \hline
 6,10,1     & 6,4,2     & 6,5,2     & 6,3,3   & 7,6,1   & 7,7,1     & 7,8,1     \\  \hline
    0          & 330       & 45        & 15      & 16,807  & 23,100    & 16,800   \\     \hline \hline
  7,9,1       & 7,10,1     & 7,5,2      & 7,6,2      & 7,7,2     & 7,4,3   & 8,7,1 \\   \hline
7,770       & 2,331      & 4,305      & 1,575      & 210       & 315       & 262,144 \\  \hline \hline
8,8,1   & 8,9,1     & 8,10,1      & 8,6,2       & 8,7,2      & 8,8,2 & 8,9,2  \\      \hline
513,240 & 555,520   & 412,440     & 66,248      & 45,360     & 15,435     & 2,940  \\     \hline \hline
 8,10,2    & 8,5,3     & 8,6,3   & 8,4,4   & 9,8,1     & 9,9,1 & 9,10,1  \\    \hline
280       & 5,880     & 630     & 105     & 4,782,969  & 12,551,112  & 18,601,380  \\  \hline   \hline
9,7,2      & 9,8,2      & 9,9,2      & 9,10,2    & 9,6,3     & 9,7,3  & 9,8,3 \\ \hline
1,183,644  & 1,287,090  & 768,600    & 309,960   & 115,290   & 34,020  & 3,780  \\    \hline \hline
9,5,4     &10,9,1      & 10,10,1     & 10,8,2     & 10,9,2     & 10,10,2  & 10,7,3 \\ \hline
3,780     & 100,000,000 & 336,853,440 & 24,170,310 & 37,948,680 & 34,146,000 & 2,467,080 \\  \hline \hline
10,8,3    & 10,9,3  & 10,10,3 & 10,6,4   & 10,7,4      & 10,5,5    &   \\   \hline
 1,379,700 & 392,175 & 66,150  & 107,100   & 9,450       & 945     &  \\ \hline
\end{tabular}
\end{table}
A missing triple $(n,k,\nu)$ (such as $(3,1,1)$) means there are no bipartite graphs of order $n$, size $k$ with $\nu$ connected components.
\bibliographystyle{alpha}
\bibliography{bipartite}

\end{document}